\documentclass{article}
\usepackage{amsmath,amssymb,amsbsy,amsfonts,
amsthm,latexsym,amsopn,amstext,
amsxtra,euscript,amscd,url}
\usepackage[utf8]{inputenc}
\newtheorem{theorem}{Theorem}[section]
\newtheorem{lemma}[theorem]{Lemma}

\newtheorem{remark}[theorem]{Remark}

\newtheorem{proposition}[theorem]{Proposition}
\newtheorem{corollary}[theorem]{Corollary}

\newtheorem{rmk}[theorem]{Remark}

\DeclareMathOperator{\ord}{ord}

\def\\{\cr}
\def\({\left(}
\def\){\right)}
\def\[{\left[}
\def\]{\right]}
\def\<{\langle}
\def\>{\rangle}

\def\li{\operatorname{li}}

\def\cP{\mathcal P}

\def\eps{\varepsilon}

\def\Z{\mathbb{Z}}

\def\Q{\mathbb{Q}}
\def\ka{\kappa}

\def\notdivides{\mathrel{\kern-3pt\not\!\kern3.5pt\bigm|}}

2
\makeatletter
\renewcommand*\l@section[2]{%
  \ifnum \c@tocdepth >\z@
    \addpenalty\@secpenalty
    \addvspace{0.2em \@plus\p@}%
    \setlength\@tempdima{1.5em}%
    \begingroup
      \parindent \z@ \rightskip \@pnumwidth
      \parfillskip -\@pnumwidth
      \leavevmode \bfseries
      \advance\leftskip\@tempdima
      \hskip -\leftskip
      #1\nobreak\hfil \nobreak\hb@xt@\@pnumwidth{\hss #2}\par
\endgroup \fi}
 
\makeatother

\numberwithin{equation}{section}

1

\title{Random ordering in modulus of
consecutive Hecke eigenvalues of 
primitive forms}

\author{
{\sc Yuri ~F.~Bilu}, 
{\sc Jean-Marc ~Deshouillers},\\ 
{\sc  Sanoli ~Gun}, 
{\sc Florian~Luca}
}

\date{}

\begin{document}

\hfuzz 3.1pt

\maketitle

\begin{abstract}
Let $\tau(\cdot)$ be the classical Ramanujan $\tau$-function and 
let~$k$ be a positive integer such that ${\tau(n)\ne 0}$ for 
${1\le n\le k/2}$. (This is known to be true for ${k < 10^{23}}$, and, 
conjecturally, for all~$k$.) Further, let~$\sigma$ be a permutation 
of the set ${\{1,...,k\}}$. We show that there exist infinitely 
many positive integers~$m$ such that 
${|\tau(m+\sigma(1))|< |\tau(m+\sigma(2))|<...<|\tau(m+\sigma(k))|}$. 
We also obtain a similar result for Hecke eigenvalues of primitive
forms of square-free level.

\bigskip


\end{abstract}

\section{Introduction}
\label{sintr}
Throughout the article a \textsl{primitive form} of 
weight~$\kappa$ and level~$N$ means a
holomorphic cusp form of weight~$\kappa$ for $\Gamma_0(N)$ with 
the trivial character which is also a normalized Hecke eigenform 
for all Hecke operators as well of all 
Atkin-Lehner involutions (see page 29 of \cite{KO} for more details). 
Throughout the paper, we will also assume that $N$ is square-free.
 A \textsl{non-CM  primitive form} is an abbreviation 
for ``primitive form without Complex Multiplication''.  

Let~$f$ be a  primitive form and
$$
f(z) : = \sum_{n \ge 1} a_f(n) q^n, \qquad q=e^{2\pi iz}
$$ 
be its Fourier expansion at $i\infty$. 
In particular, if~$f$ is of weight ${\kappa=12}$ and level ${N=1}$
then
$$
f (z) 
= \Delta(z) 
:=\sum_{n\ge 1} \tau(n)q^n=q\prod_{\ell\ge 1} (1-q^\ell)^{24},
$$
where $\tau(n)$ is the classical Ramanujan function.

It is well known that the Fourier-coefficients $a_f(n)$'s of 
any such primitive form $f$
are totally real algebraic numbers.
There are quite a few results demonstrating ``random'' behavior 
of the signs of $\tau(n)$, or, more generally, the coefficients of a 
general primitive forms; see, for instance, \cite{GS12, GKR15,
KS09, Ma12, MR15} and the references therein.  For instance, 
Matomäki and Radziwiłł~\cite{MR15}  have shown that the 
non-zero coefficients of primitive forms for $\Gamma_0(1)$  
are positive and negative with the same frequency.
They also show that for large enough $x$, the number of sign
changes in the sequence $\{ a_f(n)\}_{ n\le x}$ 
is of the order of magnitude
$$
\# \{ n \le x :  a_f(n) \ne 0 \} 
\asymp x 
\prod_{\genfrac{}{}{0pt}{}{p \le x}{a_f(p) = 0}} \left(1 - \frac{1}{p}\right).
$$

In  this paper, we work in a different direction, and study the 
behavior of absolute values of non-zero coefficients. 
Classical results of Rankin~\cite{Ra39,Ra70} 
$$
\sum_{n\le x}|\tau(n)|^2\asymp x^{12} \quad \text{and} 
\quad \limsup_{n\to\infty}\frac{|\tau(n)|}{n^{11/2}}=+\infty 
$$
imply
that the sequence  $|\tau(n)|$ is not ultimately monotonic; 
in other words, each of the inequalities
$$
|\tau(m+1)|<|\tau(m+2)|, \qquad |\tau(m+2)|<|\tau(m+1)|
$$
holds for infinitely many~$m$. In this article we obtain 
(as a special case 
of a more general result) a similar statement for more 
than two consecutive values of~$\tau$.

\begin{theorem}
\label{thtau}
Let~$k$ be a positive integer such that
\begin{equation}
\label{etauhalf}
\tau(n)\ne 0 \qquad (1\le n\le k/2).
\end{equation}
Then for every permutation~$\sigma$ of the set 
${\{1, \ldots, k\}}$, there exist 
infinitely many positive integers~$m$ such that 
\begin{equation}
\label{etausig}
0<|\tau(m+\sigma(1))|<|\tau(m+\sigma(2))|<\cdots< |\tau(m+\sigma(k))|. 
\end{equation}
\end{theorem}

In fact,  existence of at least one~$m$ satisfying \eqref{etausig} 
implies \eqref{etauhalf}, see Theorem~\ref{thdes} below; 
in other words, \eqref{etauhalf} is a necessary and 
sufficient condition for \eqref{etausig} to happen infinitely often. 

It is known \cite[Theorem~1.4]{ZY15} that   ${\tau(n) \ne 0}$ when
$$
n\le982149821766199295999\approx 9\cdot 10^{20}.
$$
We also refer to the Corollary~1.2 of the unpublished 
article~\cite{DHZ13}, 
which claims that ${\tau(n) \ne 0}$ for all
${n\le 816212624008487344127999\approx8\cdot10^{23}}$.

According to a famous conjecture of Lehmer,  
${\tau(n)\ne 0}$ for all $n$.
If this  conjecture holds true, then 
Theorem~\ref{thtau} applies  to all~$k$.

In this context, one has another famous conjecture
known as Maeda's conjecture. Let $T_n(x)$
be the characteristic polynomial of the $n$-th Hecke operator $T_n$
acting on the vector space of cusp forms of
weight $\kappa$ and level $1$, denoted $S_{\kappa}(1)$.
It is well known that $T_n(x)$ is a polynomial with integer
coefficients. Maeda~\cite{HM} conjectured that 
for any non-zero natural number $n$,
the polynomial $T_n(x)$ is irreducible over $\Q$ 
with Galois group ${\mathfrak{S}}_d$,
where $d$ is the dimension of $S_{\kappa}(1)$ 
and ${\mathfrak{S}}_d$ is the symmetric group on $d$ symbols. 
If the dimension $d$ of $S_{\kappa}(1)$ is strictly
greater than one and Maeda's conjecture is true, then 
Theorem~1.1 applies to all~$k$.
However,  Maeda's conjecture does not
imply Lehmer's conjecture.

\newpage

Our principal result is the following general theorem. 
\begin{theorem}
\label{thnew}
Let $f_1, \ldots, f_k$ be primitive forms 
of square-free levels, 
not necessarily of same weights, and  
${\nu_1, \ldots, \nu_k}$  be distinct positive 
integers such that 
\begin{equation*}
a_{f_i}(\nu_i)\ne 0 \qquad (1\le i\le k).
\end{equation*} 
Then there exist infinitely many positive 
integers~$m$ such that 
\begin{equation}
\label{etausnu}
0<|\lambda_{f_1}(m+\nu_1)|<|\lambda_{f_2}(m+\nu_2)|<\cdots
 <|\lambda_{f_k}(m+\nu_k)|,
\end{equation} 
where
$\lambda_{f_i}(n) = a_{f_i}(n)/n^{(\kappa_i -1)/2}$ for any positive
integer $n$ and $1 \le i \le k$.
\end{theorem}

In fact, we prove (see Remark~\ref{requant}) that for 
sufficiently large positive 
number~$x$, there are at least 
${c x/(\log x)^k}$ positive integers ${m\le x}$ 
satisfying \eqref{etausnu}. 
Here ${c>0}$ depends on ${f_1, \ldots, f_k,\nu_1, \ldots, \nu_k}$ 
and ``sufficiently 
large'' translates as ``exceeding a certain quantity 
depending 
on ${f_1, \ldots, f_k, \nu_1, \ldots, \nu_k}$'' .

It is clear from our proof that, when the forms ${f_1, \ldots, f_k}$ have  equal weights,  inequality~\eqref{etausnu}  holds true with $a_{f_i}(\cdot)$ instead of $\lambda_{f_i}(\cdot)$. An interesting special case occurs when 
${f_1=\cdots=f_k=f}$.

\begin{theorem}
\label{tcorold}
Let~$f$ be a primitive form of square-free level
and ${\nu_1, \ldots, \nu_k}$ be distinct positive 
integers such that 
\begin{equation*}
a_{f}(\nu_i)\ne 0 \qquad (1\le i\le k).
\end{equation*} 
Then there exist infinitely many positive 
integers~$m$ such that 
\begin{equation}
\label{etausnu1}
0<|a_{f}(m+\nu_1)|<|a_{f}(m+\nu_2)|<\cdots <|a_{f}(m+\nu_k)|.
\end{equation} 
In particular, if $k$ is a positive integer such that
\begin{equation}
\label{etnez}
a_f(n)\ne 0 \qquad (1\le n\le k),
\end{equation} 
then for every permutation~$\sigma$ of the set 
${\{1, \ldots, k\}}$, there exist 
infinitely many positive integers~$m$ such that 
\begin{equation}
\label{eq:1}
0<|a_f(m+\sigma(1))|<|a_f(m+\sigma(2))|<\cdots 
<|a_f(m+\sigma(k))|.
\end{equation} 
\end{theorem}

In fact,  one can do even better: to give a  necessary and 
sufficient condition for having \eqref{eq:1} infinitely often.

\begin{theorem}
\label{thdes}
Let~$k$ be a positive integer. Then for a 
primitive form~$f$ of square-free level
the following three conditions  are equivalent.  
 
\begin{enumerate}
\item
\label{inezhalf}
We have 
\begin{equation}
\label{etnezhalf}
a_f(n)\ne 0 \qquad (1\le n\le k/2).
\end{equation} 

\item
\label{inezk}
For some positive integer~$\nu$ we have 
\begin{equation}
\label{etnezk}
a_f(\nu+n)\ne 0 \qquad (1\le n\le k).
\end{equation} 

\item
\label{iperm}
For every permutation~$\sigma$ of the set 
${\{1, \ldots, k\}}$, there exist 
infinitely many positive integers~$m$ such that 
\begin{equation*}
0<|a_f(m+\sigma(1))|<|a_f(m+\sigma(2))|<\cdots 
<|a_f(m+\sigma(k))|.
\end{equation*} 
\end{enumerate}
\end{theorem}

Theorem~\ref{thtau} follows from this theorem if 
we take ${f=\Delta}$. 

\begin{remark}
\label{rnocm}
Since it is known that there are no primitive forms
with Complex Multiplications for square-free level
(see \cite{KR}, Section~3 and \cite{KR1}, Theorem~3.9),
the primitive forms considered by us are
necessarily non-CM. 
\end{remark}

Techniques of the proofs rely on elementary arguments, 
sieve methods (Brun's sieve, the Bombieri-Vinogradov Theorem), 
and validity of the Sato-Tate conjecture for non-CM
forms. Similar results may be expected for Maass forms,
but for the time being, we do not even know that $a_{f_i}(\nu_i)\ne 0$
for a positive proportion of $\nu_i$ though it is expected
to be true for Maass forms of eigenvalue strictly greater
that $1/4$. Also the analogs of the Ramanujan-Petersson 
and the Sato-Tate conjectures are not known to be true.  

For our construction of special values of $m$ for which \eqref{etausnu} 
and \eqref{etausnu1} holds, we
choose by force the small prime factors of the $m + \nu_i$ so that their
contribution ensures the wished ordering of the  $|\lambda_f(m +\nu_i) |$
or $|a_f(m +\nu_i) |$, with a
little margin, and we expect that the larger prime factors will contribute only
within the margin. The first step is to eliminate, thanks to the Fundamental Lemma of the Sieve Theory,
the midsize primes.  Only the large primes remain, which are essentially bounded in
number. To keep control of their contribution, we need to avoid the prime powers,
which is easily done (Section~\ref{ssubmhoprime}) since we have an explicit bound for the sum of
the inverse of the squares larger than $z$. We are happy that Deligne-Ramanujan
ensures that the contribution of the large primes is never very large, but we have
to take care of those large primes for which $|\lambda_f(p)|$ or $|a_f(p)|$ 
is small; thanks to
our colleagues who worked hard to give right to Sato-Tate (\cite{BLGHT11},~\cite{CHT08} and~\cite{HSBT10}), we know that those primes are not too numerous;
but we do not have explicit bounds as we have for the prime powers. This is where we
need to trade the sifting level, which can be small for the sieve part, but which
has to be large enough to insure that the contribution of the large ``bad" primes is
small. 

The article is organized as follows. In Section~\ref{scoefs}, 
we briefly 
review the properties of the coefficients of  primitive forms 
used in the sequel. 
In Sections~\ref{ssieving} and~\ref{savoiding}, we obtain 
two sieving results instrumental for the proof of 
Theorem~\ref{thnew}, Theorem~\ref{tcorold} and 
Theorem~\ref{thdes}. Finally, these 
theorems are proved in Sections~\ref{sproof} 
and~\ref{sdes}, respectively. 

\subsection{Conventions}
\label{ssconv}

Unless the contrary is stated explicitly:
\begin{itemize}

\item
$p$ (with or without indices) denotes a prime number;

\item
$\ka$ denotes a positive even integer;

\item 
$i,j,k, \ell, m$ (with or without indices) denote positive integers;

\item
$n$ (with or without indices) denotes a non-negative integer;

\item
$d$ (with or without indices) denotes a square-free positive integer; 

\item 
$\eps,\delta$ denote real numbers satisfying ${0<\eps,\delta\le 1/2}$;

\item
$x,y,z,t$ denote real numbers satisfying ${x,y,z,t\ge 2}$. 
\end{itemize}

\section{Hecke eigenvalues of  primitive forms}
\label{scoefs}
In this section, we list some well-known properties of 
the Hecke eigenvalues of primitive forms which will be used 
in the proof of Theorem~\ref{thnew} and Theorem~\ref{tcorold}.

First of all, the Hecke eigenvalues $a_f(n)$ are multiplicative:
\begin{equation}
\label{emult}
a_f(mn)=a_f(m)a_f(n) \qquad (m, n\ge 1, \quad \gcd(m,n)=1).
\end{equation}
Furthermore, the values of~$a_f$ at prime powers satisfy the 
following recurrence relations 
\begin{eqnarray}
\label{erecu}
a_f(p^{\ell+1})
&=&
a_f(p)^{\ell +1} 
 \quad \text{if  } p|N   \\
a_f(p^{\ell+1})
&=&
a_f(p)a_f(p^\ell)-p^{\ka -1}a_f(p^{\ell-1})
\text{  if } (p, N)=1, 
\quad (\ell=1,2,\ldots), \nonumber \end{eqnarray}
where~$\kappa$ is the weight of~$f$.

Both \eqref{emult} and \eqref{erecu} were conjectured by 
Ramanujan when ${f = \Delta}$ and 
proved by Mordell~\cite{Mo17}. 
Proofs can  be found in many sources; see, 
for instance \cite[Proposition 5.8.5.]{DS05}.

A much deeper result is the upper bound
\begin{equation}
\label{edel}
|a_f(p)|\le 2p^{(\ka-1)/2}.
\end{equation}
It was also conjectured by Ramanujan when $f = \Delta$ and 
proved by Deligne \cite[Théorème~8.2]{De74}. 
Equivalently, the polynomial ${T^2 - a_f(p)T + p^{\ka -1}}$ 
can not have distinct real roots. 
Hence we may write the roots as
\begin{equation}\label{econj}
\alpha_p
=p^{(\ka-1)/2}e^{i\theta_p}, 
\qquad 
\bar\alpha_p
=p^{(\ka-1)/2}e^{-i\theta_p}, 
\end{equation}
with ${\theta_p\in [0,\pi]}$. As before,
we shall write 
$$
\lambda_{f}(n) = a_{f}(n)/n^{(\kappa -1)/2}
$$ 
for any positive integer $n$. If ${\theta_p\ne 0,\pi}$ (that is, 
${\lambda_f(p)\ne \pm 2}$) then 
\begin{equation}\label{esines}
\lambda_f(p^\ell)
= \frac{\sin(\ell+1)\theta_p}{\sin\theta_p}.
\end{equation}
We may add for completeness that
\begin{equation}
\label{ezeropi}
\lambda_f(p^\ell)=
\begin{cases}
(\ell+1), & \theta_p=0, \\
(-1)^\ell (\ell+1) , & \theta_p=\pi. \\
\end{cases}
\end{equation}

Another very deep result is the \textit{Sato-Tate conjecture}, 
proved recently by Barnet-Lamb, Geraghty, Harris and 
Taylor \cite[Theorem~B]{BLGHT11} (see 
also \cite{{CHT08},{HSBT10}}). 
A convenient way to express it is to use the notion of \emph{relative density} of a
set of primes: we say that a set $\mathcal{P}$ of primes has the relative density
$\delta(\mathcal{P})$ (\emph{resp}. the relative upper density
$\bar{\delta}(\mathcal{P}))$ if
\begin{equation}
\label{reldens}
\delta(\mathcal{P}) = \lim \frac{\#(\mathcal{P}\cap[1, x])}{\pi(x)} \;
\left(\text{\emph{resp}. }\; \bar{\delta}(\mathcal{P}) = \limsup
\frac{\#(\mathcal{P}\cap[1, x])}{\pi(x)}\right), 
\end{equation}
as $x \rightarrow +\infty$, where $\pi(x)$ denotes the number of primes up to $x$.

The above-mentioned result states that, for a non-CM primitive 
form~$f$, the numbers ${\lambda_f(p)}$ are
equi-distributed in the interval ${[-2, 2]}$ with respect 
to the Sato-Tate measure 
${(1/\pi)\sqrt{1- t^2/4}~dt}.$ 
This means that for ${-2\le a\le b\le 2}$, we have
\begin{equation}
\label{esatotate}
\delta\left(\{p \colon \lambda_f(p) \in [a, b] \}\right)= 
\frac1\pi\int_a^b\sqrt{1-\frac{t^2}{4}}dt.
\end{equation}

An immediate consequence of this and Remark~\ref{rnocm} is the following statement.

\begin{proposition}
\label{prsatotate}
Let $f$ be a primitive form of square-free level. 
Then the following holds.
\begin{enumerate}
\item
\label{ieps}
The relative density of the set of primes~$p$ 
such that $\lambda_f(p)$ 
belongs to a given interval of 
length~$2\eps$ does not exceed~$\eps$.

\item
\label{izero}
In particular, the relative density of primes~$p$ 
such that ${\lambda_f(p)=0}$ or ${\pm 2}$  is~$0$. 
\end{enumerate}
\end{proposition}

We notice that the formulation \ref{ieps} is convenient to 
use for our purpose, but our argument could be adapted to 
the weaker condition
$$
\bar{\delta}\left(\{p \colon \lambda_f(p) 
\in [-\varepsilon, +\varepsilon] \}\right)
\rightarrow 0 \text{ as } \varepsilon \rightarrow 0.
$$
Part~\ref{izero} was well known long before 
the proof of the Sato-Tate conjecture. See Théorème~15 
in \cite[Section~7.2]{Se81} for a much more general and 
quantitatively stronger result.

Equations \eqref{esines} and \eqref{ezeropi} imply that
${\lambda_f(p^\ell)= 0}$ for 
some~$\ell$ and $(p, N)=1$
if and only if ${\theta_p/\pi\in \Q\cap(0,1)}$. In fact, one 
knows the following result.

\begin{proposition}
\label{prtaulnzer}
Let $f$ be a primitive form of square-free level. 
Then for all but finitely many primes~$p$ 
we have either  ${\lambda_f(p) \in\{ 0, \pm 2\}}$ 
or ${\theta_p/\pi \notin \Q}$. 
\end{proposition}

For the proof, see \cite[Lemma~2.5]{MM07} (see also 
\cite[Lemma~2.2]{KRW07}).  

One may remark that if~$f$ is of weight ${\kappa\ge 4}$ then this 
holds for all~$p$ with $(p,N)=1$
 without exception, see \cite[Proposition~2.4]{MM07}.

\section{Sieving} 
\label{ssieving}

In this section, we establish a sieving result instrumental 
for the proof of Theorem~\ref{thnew} and 
Theorem~\ref{tcorold}. The integer~$m$ in this 
section is not necessarily positive; it  can be any integer: positive, 
negative or~$0$. The other conventions made in 
Subsection~\ref{ssconv} remain intact.

\subsection{The Sieving Theorem}
\label{sssievtheo}

Let~$\Sigma$ be a finite set of prime numbers. We call ${m\in \Z}$

\begin{itemize}
\item
\textit{$\Sigma$-unit}, if all its prime divisors belong to~$\Sigma$;

\item
\textit{$\Sigma$-square-free}, if~$m$ is a product of a 
$\Sigma$-unit and a square-free integer. 
\end{itemize}
Also, for ${z\ge 2}$ we define
\begin{equation}
\label{epsigmaz}
P_\Sigma(z)
=
\prod_{\genfrac{}{}{0pt}{}{p< z}{p\notin \Sigma}}p.
\end{equation}

Now let
${a_1,\ldots, a_k, b_1, \ldots, b_k\in \Z}$ be integers satisfying 
\begin{align}
\label{enondegeni}
&a_i \ne 0,\quad \gcd(a_i,b_i)=1 &&(i=1, \ldots, k), \\
\label{enondegenij}
&a_ib_j-a_jb_i \ne 0 &&(1\le i< j\le k). 
\end{align}
We consider linear forms ${L_i(n)=a_in+b_i}$, and 
for ${x\ge z\ge 2}$ we set
\begin{equation}
\label{eaxz}
\Omega(x,z)= \left\{n: 1\le n\le x, \quad 
\gcd\bigl(L_1(n)\cdots L_k(n), P_\Sigma(z)\bigr)=1\right\}.
\end{equation}
Finally, we let
{\small
\begin{equation}\label{mho1}
 \Omega_1(x,z) =  \{n \in \Omega(x,z) \colon L_1(n), \ldots, L_k(n) 
 \text{ are } \Sigma\text{-square-free composite numbers}\}.
\end{equation}}%

The principal result of this section is the following theorem.

\begin{theorem}
\label{thsieving}
Assume that~$\Sigma$ contains all the primes ${p\le 2k}$, 
all the prime divisors of every~$a_i$, and all the prime divisors 
of every ${a_ib_j-a_jb_i}$ with ${i\ne j}$. In other words, 
we assume that 
\begin{equation}
\label{esunit}
\text{$(2k)!\prod_{i=1}^ka_i\prod_{1\le i<j\le k}
(a_ib_j-a_jb_i)$}
\end{equation}
is a $\Sigma$-unit. Then there exist real numbers 
${\eta, c_1\in (0, 1/2] }$, depending only on~$k$ and on the 
cardinality\footnote{Indicating dependence on~$k$ 
here is somewhat useless, 
because our hypothesis implies that~$k$ is bounded in terms 
of~$\#\Sigma$.}~$\#\Sigma$ (but not on~$\Sigma$ itself, neither 
on the integers $a_i$ and $b_i$),
and ${z_1\ge 2}$ depending on ${a_1, \ldots, a_k, b_1,\ldots, b_k}$,  
such that the following holds. For any~$x$ and~$z$, satisfying
${x^{\eta}\ge z\ge z_1}$
we have 
$$
\#\Omega_1(x,z)\ge c_1\frac{x}{(\log z)^{k}}.
$$ 
\end{theorem}

The first step in the proof of Theorem~\ref{thsieving} is to obtain 
a lower bound for $\#\Omega(x, z)$, i.e. we wish to get a lower 
bound for the number of integers up to $x$ for which the 
product $L_1(n)\cdots L_k(n)$ has no prime factor up to 
$x^{\eta}$ except from a finite given set $\Sigma$; 
in other words, we are interested in sieving out the 
prime factors less than $x^{\eta}$ except those from 
$\Sigma$, when $\eta$ is sufficiently small: the adapted 
tool for this situation is called the Fundamental 
Lemma, cf.~\cite{opera}, Section~6.5 or \cite{HR74}, 
Section~2.8. Looking more carefully at \cite{HR74}, 
we see that, with the exception of $\Sigma$, 
Theorem~2.6, p.~85, is very close to what we are looking 
for. In Section~\ref{sslbmho}, we shall state and prove the variant of 
Theorem 2.6 we need. We obtain a lower bound 
of the order $x(\log z)^{-k}$.

In the second step, we need to exclude the cases when
at least one of the quantities $L_i(n)$ is a prime number. Assume 
for example that $L_k(n) = n$, we see that 
Theorem 2.6' of~\cite{HR74}, p.~87, applied 
to the product ${L_1(n)\cdots L_{k-1}(n)}$ (with ${k-1}$ 
instead of~$k$) is, again with the exception of the primes 
from $\Sigma$, very close to what we are looking for. 
In Section~\ref{ssubmhoprime}, we shall state and prove the variant of 
Theorem 2.6' we need. We obtain an upper bound of 
the order $x(\log z)^{-k+1}(\log x)^{-1}$, which is smaller 
than the lower bound from the first step, as soon as $z$ 
is sufficiently small a power of $x$, i.e. as soon 
as $\eta$ is small enough.

The last step consists in sieving out the 
elements of $\Omega(x, z)$ 
divisible by the square of some large prime; 
the key ingredient is the convergence of the series of 
the inverses of the squares. 
This step is performed in Section~\ref{ssnosqfree}.

Finally, in Section~\ref{proofthm3.1} we prove 
Theorem~\ref{thsieving}.

We start by giving in Section~\ref{ssarithquant} some definition 
and evaluation of some arithmetic quantities.

\subsection{Some arithmetic preliminaries}
\label{ssarithquant}

In the remaining part of Section~\ref{ssieving},  
unless the contrary is explicitly stated, the 
constants implied by the notation $O(\cdot)$, 
$\ll$, $\gg$ or\footnote{We use ${A\asymp B}$ 
as a shortcut for ${A\ll B\ll A}$.} $\asymp$, may 
depend only on~$k$. The same convention 
applies to the the constants implied by the 
expressions like  ``sufficiently large''. 

In order to avoid a conflict of notation between~\cite{HR74} and the general use, we follow, in 
Sections~\ref{ssarithquant},~\ref{sslbmho} and~\ref{ssubmhoprime},  
the use of~\cite{HR74} and denote by $\nu(d)$ the 
number of distinct prime factors of the integer $d$.

We keep the notation of Section~\ref{sssievtheo} and let $\ell \in \{k-1, k\}$,
\begin{equation}\label{defF}
F_{\ell}(n) = L_1(n) \cdots  L_{\ell}(n).
\end{equation}
Let~$\rho_{\ell}$ be the multiplicative function supported 
on the square-free numbers and such that 
$$
\rho_{\ell}(p) =
\begin{cases}
\ell,& p\notin \Sigma ,\\
0, & p\in \Sigma. 
\end{cases}
$$ 
For $z \ge 2$, we let
\begin{eqnarray}\label{evzdef}
W_{\ell}(z) 
&=& 
\prod_{p \le z} \left(1 - \frac{\rho_{\ell}(p)}{p}\right) 
= 
\prod_{p \mid \mathcal{P}_{\Sigma}(z)}\left(1 - \frac{\ell}{p}\right),\\
\label{evpzdef}
W_{\ell}^{(*)}(z) &=& \prod_{p \le z} 
\left(1 - \frac{\rho_{\ell}(p)}{p-1}\right) = 
\prod_{p \mid \mathcal{P}_{\Sigma}(z)}\left(1 - \frac{\ell}{p-1}\right),
\end{eqnarray}
with the usual convention that an empty product is equal to $1$.

Our assumption~\eqref{esunit} implies that
the congruence 
$$
F_{\ell}(n) = L_1(n)\cdots L_{\ell}(n)\equiv 0\pmod p
$$
has exactly $\rho_{\ell}(p)=\ell$ solutions for any prime 
$p$ which does not belong to the set~$\Sigma$; 
moreover, all those solutions are non-zero.
Thus, the congruence 
\begin{equation}\label{Fmodd}
F_{\ell}(n) = L_1(n)\cdots L_{\ell}(n)\equiv 0\pmod d
\end{equation}
has exactly ${\rho_{\ell}(d)=\ell^{\nu(d)}}$ solutions 
for any square-free~$d$ having no prime divisor 
from the set~$\Sigma$; moreover all those solutions 
are coprime with $d$. This implies 
\begin{equation}\label{Rd}
\mid\#\{n \in [1, x] \colon F_{\ell}(n) \equiv 0 \pmod d\} 
- x\rho_{\ell}(d)/d \mid\,  \le \,  \rho_{\ell}(d) = \ell^{\nu(d)}.
\end{equation}
Since all primes ${p\le 2k}$ belong to~$\Sigma$, we have,
for all primes $p$, the estimates  
\begin{equation}
\label{erhopp}
0\le \frac{\rho_{\ell}(p)}{p} 
\le \frac{\rho_{\ell}(p)}{p-1}\le \frac12.
\end{equation}
We trivially have
\begin{equation}
\label{ewzprod}
W_{\ell}(z) \ge \prod_{2k < p\le z} \left(1-\frac {\ell}p\right).
\end{equation}
Using~\eqref{erhopp}, we get the upper bound
\begin{equation}\label{ewzstarprod}
W^{(*)}_{\ell}(z) \le 2^{\#\Sigma}
\prod_{2k<p\le z} \left(1-\frac {\ell}{p-1}\right).
\end{equation}

We also notice that Mertens' result (\cite{HW}, Theorem 429), 
easily implies that there exists constants $C(\ell)$ 
and $C^{(*)}(\ell)$ such that
\begin{equation}\label{Mertens1}
\prod_{2k< p \le z} \left(1 - \frac{\ell}{p}\right) \sim C(\ell) (\log z)^{-\ell}  
\text{ and } 
\prod_{2k< p \le z} \left(1 - \frac{\ell}{p-1}\right) \sim C^{(*)}(\ell) (\log z)^{-\ell}.
\end{equation}
The following is a fairly standard result, a proof of 
which can be found in \cite{TW}, p. 55.
\begin{equation}\label{meankomega}
\text{ As $x$ tends to infinity } \colon \sum_{n\le x}\ell^{\nu(n)} 
\sim c_{\ell} x (\log x)^{\ell-1}, 
\text{ for some positive } c_{\ell}.
\end{equation}

For $d >0$ and $a$ coprime to $d$, we 
denote by $\pi(x, d, a)$ the 
number of primes up to $x$ which are congruent 
to $a$ modulo $d$ and we let 
\begin{equation}\label{errorprime}
E(x, d, a) = \pi(x, d, a) - \frac{ \li x}{\varphi(d)} 
\; \text{ and }\; 
E(x, d) = \max_{\gcd(a, d)=1} |E(x, d, a)|.
\end{equation}
We shall use the following consequence of 
Lemma 3.5 of \cite{HR74}, p.115, which is itself a consequence 
of the Bombieri-Vinogradov Theorem and the trivial upper bound
\begin{equation}\label{Etrivial}
E(x, d) \le x/d +1.
\end{equation}
\begin{lemma}\label{lemBV}
Let $m$ be a positive integer. For any positive constant $U$, there 
exists a positive constant $C_1 = C_1(m, U)$ such that
\begin{equation}\label{eqBV}
\sum_{d < x^{1/2} (\log z)^{-C_1}} \mu^2(d) 
m^{\nu(d)}E(x,d) = O_{U, m} \left(\frac{x}{(\log z)^{U}}\right).
\end{equation}
\end{lemma}

\subsection{Sieving away small prime factors}
\label{sslbmho}

In this section, we prove the following result.

\begin{proposition}\label{propbmho}
With the above notation and assumption (\ref{esunit}), we have for $2 \le z \le x$
\begin{eqnarray}
\label{term0}
\#\Omega(x, z) = xW_k(z) (1+O(E_0(x, z))), 
\end{eqnarray}
where
\begin{eqnarray}
\label{error0}
E_0(x, z) = \exp(-u(\log u - \log \log u - \log k - 2)) + \frac{1}{\log z}, \\
\label{defu}
\text{and } 
u=
\log x / \log z. 
\phantom{mmmmmmmmmmmmmmmmm}
\end{eqnarray}
\end{proposition}

\begin{proof}
We are going to use Theorem 2.5' of \cite{HR74}, noticing 
that in the main relation, $\log x/\alpha$ is to be read 
$\log(\kappa/\alpha)$. We refer the Reader to \cite{HR74} 
for the statement of Theorem 2.5', as well as the 
notation given there. Let us write the dictionary 
between the notation from \cite{HR74} and our notation.
\begin{align}
\notag
\mathcal{A} & =  \{F_k(n) \colon 1 \le n \le x\},\\
\notag
\mathfrak{P} &= \{p \colon p \notin \Sigma\} \text{ and } 
\overline{\mathfrak{P}} = \Sigma,\\
\notag
\omega(d) &= \rho_k(d),\\
\notag
\kappa &= k,\\
\notag
X &= x,\\
\notag
U &=1,\\
\notag
\alpha &= 1,\\
\notag
R_d &= \#\{n \in [1, x] \colon F_k(n) \equiv 0 \pmod d\} - x\rho_k(d)/d. 
\end{align}

Relation (\ref{erhopp}) implies ($\Omega_1$) (p.29) with $A_1=2$.

By the definition of $\rho_k$, we have for all $p$ : $\rho_k(p) \le k$, 
which implies Relation ($\Omega_0$) 
of \cite{HR74}, p. 30, and thus (cf. Lemma 2.2 p. 52) 
Relation ($\Omega_2(\kappa)$) with $A_2=\kappa$.

Relations ($R_0$) and ($R_1(\kappa, 1)$) (defined 
in p. 64 of \cite{HR74}), with $L=1$, $A_0'=k$ and 
$C_0(U) = 2k+U-1$  come from (\ref{Rd}) 
and (\ref{meankomega}).

We notice that $S(\mathcal{A}; \mathfrak{P}, z)$ is 
$\#\Omega(x, z)$ and thus Theorem 2.5' of  \cite{HR74} 
implies our Proposition \ref{propbmho}.
\end{proof}


\subsection{Sieving away prime values}
\label{ssubmhoprime}

In this part, we are interested in evaluating the 
cardinality of the set
\begin{equation}\label{mhopstar}
\Omega^{(*)}(x, z) = \{n \le x \colon  \gcd\left(L_1(n)\cdots 
L_{k-1}(n), P_{\Sigma}(z)\right) = 1, L_k(n) \text{ prime}\},
\end{equation}
and we shall prove the following
\begin{proposition}\label{ubmhoprime}
With the above notation and assumption (\ref{esunit}), 
we have for $2 \le z \le x$
\begin{eqnarray}
\label{termp}
\#\Omega^{(*)}(x, z) = \frac{\li(|a_k|x)}{\varphi(|a_k|)}W^{(*)}_{k-1}(z) 
(1+O(E^{(*)}(x, z))), 
\end{eqnarray}
where
\begin{eqnarray}
\label{errorp}
E^{(*)}(x, z) 
= \exp(-(u/3)(\log u - \log \log u - \log (k-1) - 3)) 
+ \frac{1}{(\log z)}, \\
\label{defustar}
\text{and } 
u=\log x / \log z.
\phantom{mmmmmmmmmmmmmmmmmmmmmmm}
\end{eqnarray}
\end{proposition}

\begin{proof}
We first notice that, without loss of generality, 
changing if needed $(a_i, b_i)$ into $(-a_i, -b_i)$, we can 
assume that all the $a_i's$ are positive: this is 
what we assume in the proof.

It will be convenient to let $h = k-1$, $F_h(n) = 
L_1(n)\times\cdots\times L_h(n)$. We are again 
going to use Theorem 2.5' of \cite{HR74}. Getting 
a relation ($R_1(\kappa, \alpha)$) will be more 
challenging, but the Bombieri-Vinogradov 
inequality in the form (\ref{eqBV}) will be 
most helpful. As in the previous section, we start with our dictionary.
\begin{align}
\notag
\mathcal{A} & = \{\text{$F_h((q-b_k)/a_k)$: $q$  prime, 
$a_k + b_k \le q \le a_kx + b_k$, $q \equiv b_k \bmod {a_k}$} \},\\
\notag
\mathfrak{P} &= \{p \colon p \notin \Sigma\} \text{ and }
 \overline{\mathfrak{P}} = \Sigma,\\
\notag
\omega(d) &= d\rho_{h}(d)/\varphi(d),\\
\notag
\kappa &= h = k-1,\\
\notag
X &= \li (a_k x)/\varphi(a_k),\\
\notag
U &=1,\\
\notag
\alpha &= 1/2,\\
\notag
R_d &= \# \{a \in \mathcal{A} \colon d \mid a\} - \frac{\omega(d)}{d}X .
\end{align}
We check the validity of Relations ($\Omega_0$) 
and $(\Omega_2(\kappa))$ 
by the same argument as in Section~\ref{sslbmho}.

We notice that $R_d $ is defined in terms of the cardinality of 
$\mathcal{A}_d$; it is more convenient for us to consider, for 
$d$ having no prime divisor from $\Sigma$, the set
$$
\mathcal{B}_d = \left\{q \in [a_k + b_k, a_kx +b_k] \colon q \text{ prime}, 
q \equiv b_k\!\!\! \pmod{a_k}, d \mid F_h\left(\frac{q-b_k}{a_k}\right)\right\}
$$
which has the same cardinality as $\mathcal{A}_d$. 
By the remark concerning the solutions of (\ref{Fmodd}) 
and the fact that $d$ and $a_k$ are coprime, 
there exists a set $T_k(d) \subset (\mathbb{Z}/a_kd\mathbb{Z})^{*}$ 
with cardinality $\#T_k(d) = h^{\nu(d)}$ such that
$$
q \in \mathcal{B}_d \; \Longleftrightarrow \; 
q \in [a_k + b_k, a_kx +b_k]  \text{ and } q \in T_k(d) \pmod {a_kd}.
$$
We thus have, for $d$ having no prime factor from $\Sigma$
\begin{eqnarray}
\notag
\# \mathcal{A}_d &=&\# \mathcal{B}_d = 
\quad \sum_{t \in T_k(d)} \left(\pi(a_kx, a_k d, t) + O(1)\right)\\
\notag
&=& h^{\nu(d)} \left(\frac{\li(a_k x)}{\varphi(a_k d)}\right) 
+ O\left(h^{\nu(d)} (E(a_kx, a_k d)  + 1) \right)\\
\notag
&=& \frac{h^{\nu(d)}}{\varphi(d)} X
 + O\left(h^{\nu(d)} \left(E(a_kx, a_k d)  + 1\right) \right),
\end{eqnarray}
which implies
\begin{equation}\label{Rdprime}
R_d  = O\left(h^{\nu(d)} \left(E(a_kx, a_k d)  + 1\right) \right).
\end{equation}
Relation ($R_0$) comes from the previous relation,  the trivial 
upper bound \linebreak
$E(a_k x, a_kd) \le x/d + 1$ and the definition of $X$.

Relation ($R(\kappa, 1/2)$) comes from Lemma \ref{lemBV} 
and Relation (\ref{meankomega}).

We can now apply Theorem 2.5' of \cite{HR74} and get 
Proposition \ref{ubmhoprime} with a slightly better constant and 
$u = \log X / \log z$. It is more convenient for us to state the 
result in terms of $u = \log x / \log z$.
\end{proof}

\subsection{Sieving away non-squarefree values}
\label{ssnosqfree}
We also want to count~$n$ such that  $L_i(n)$ 
is not $\Sigma$-squarefree.  
This is relatively easy. Set
$$
M=\max\{|a_1|, \ldots, |a_k|, |b_1|, \ldots, |b_k|\}. 
$$

\begin{proposition}
\label{pnosqfr}
In the set-up of Theorem~\ref{thsieving},  for ${x\ge z\ge 2}$  
the set $\Omega(x,z)$ has at most 
$$
kM\frac{x+1}{z-1}+k\sqrt{Mx+M}
$$ 
elements~$n$ such that $L_i(n)$ is not $\Sigma$-squarefree for some~$i$. 
\end{proposition}

\begin{proof}
If $L_i(n)$ is not $\Sigma$-squarefree for some 
${n\in \Omega(x,z)}$, then ${p^2\mid L_i(n)}$ for 
some ${p\ge z}$. For a fixed~$p$ and~$i$, the 
number of positive integers~$n$ with the property 
${p^2\mid L_i(n)}$ does not exceed ${(|a_i|x+|b_i|)/p^2}+1$. 
Summing up over all ${p\ge z}$ and ${i=1,\ldots, k}$, 
we estimate the total number of ${n\in \Omega(x,z)}$ 
such that some $L_i(n)$ is not $\Sigma$-squarefree as
$$
k(Mx+M)\sum_{p\ge z}\frac1{p^2}+k\pi\left({\sqrt{Mx+M}}\right).
$$
The infinite sum above is bounded by ${1/(z-1)}$, 
whence the result. 
\end{proof}

\subsection{Proof of Theorem~\ref{thsieving}}
\label{proofthm3.1}

We are now ready to prove Theorem~\ref{thsieving}.  

The Reader will easily check that one can find 
constants $c_1, z_1$ and $\eta$ satisfying the properties 
required in the statement of Theorem \ref{thsieving} 
such that the following inequalities are valid for any 
real numbers $x$ and $z$ satisfying $x^{\eta} \ge z \ge z_1$.

By Proposition \ref{propbmho}, (\ref{ewzprod}) and (\ref{Mertens1}), one has
\begin{equation}\label{minomho}
\#\Omega(x, z) \ge (1/2) x W_k(z) \ge (1/4)C(k)x (\log x)^{-k} \ge 3c_1x (\log x)^{-k}.
\end{equation}

Let us denote by $\Omega^{prime}(x,z)$ the set of the 
elements $n$ in $\Omega(x,z)$ for which one of the values 
$L_i(n)$ is prime; applying Proposition \ref{ubmhoprime} $k$ times,  
(\ref{ewzstarprod}) and (\ref{Mertens1}), we obtain
\begin{equation}\label{majomhoprime}
\#\Omega^{prime}(x,z) \le 2 k \li\left((\max_i |a_i| )x \right) 
W_{k-1}^{(*)}(z) \le c_1x (\log x)^{-k}.
\end{equation}

Let us denote by $\Omega^{square}(x,z)$ the set of the 
elements $n$ in $\Omega(x,z)$ for which one of the 
values $L_i(n)$ is not $\Sigma$-squarefree. Proposition 
\ref{pnosqfr} tells us that we have
\begin{equation}\label{majomhosq}
\#\Omega^{square}(x,z) \le kM\frac{x+1}{z-1}+k\sqrt{Mx+M} \le c_1x (\log x)^{-k}.
\end{equation}

We have 
\begin{equation}\label{thefinal}
\#\Omega_1(x, z) \ge \#\Omega(x, z) - \#\Omega^{prime}(x,z) - \#\Omega^{square}(x,z) 
\end{equation}
and Theorem \ref{thsieving} comes from 
(\ref{thefinal}), (\ref{minomho}), (\ref{majomhoprime}) and (\ref{majomhosq}).

\section{Avoiding Prime Factors from a Sparse Set}
\label{savoiding}

In this section, we further refine the set $\Omega_1(x,z)$ 
constructed in Theorem~\ref{thsieving}, showing 
that it has ``many'' 
elements~$n$  such that $L_1(n)\cdots L_k(n)$ has 
no prime divisors 
in a ``sufficiently sparse'' set of primes. We will have to 
impose an additional 
assumption: every prime from~$\Sigma$ divides every~$a_i$.  
Probably the 
statement holds true without this  assumption, but 
imposing it will facilitate 
the proof, and the result we obtain will suffice for us.

Given an infinite set of primes $\cP$, let 
${\pi_\cP(x)=\#(\cP\cap [0,x])}$ and
$\bar\delta(\cP)$ be the relative upper density 
as defined in \ref{reldens}. Also let $L_1(n), \ldots, L_k(n)$ 
and the finite set~$\Sigma$ be 
as in Subsection~\ref{sssievtheo}.

\begin{theorem}
\label{thavoiding}
Assume the hypothesis of Theorem~\ref{thsieving}. 
Moreover, assume that 
\begin{equation}
\label{esigdiva}
\text{every~$a_i$ is divisible by every  prime from~$\Sigma$}. 
\end{equation}
Let~$\eta$ be the number as in Theorem~\ref{thsieving}. 
Then there exists 
${\eps\in (0, 1/2] }$, depending only on~$k$ and 
on $\#\Sigma$, such 
that the following holds. For any set~$\cP$ of primes with 
${\bar\delta(\cP)\le \eps}$, there exists 
 ${x_0\ge 2}$ depending on ${a_1, \ldots, a_k, b_1,\ldots, b_k}$  
and on the set~$\cP$, such that for ${x\ge x_0}$ 
at least half of the elements~$n$ of the set 
$\Omega_1(x,x^\eta)$ have the property
$$
p\nmid L_1(n)\cdots L_k(n)\qquad (p\in \cP). 
$$
\end{theorem}

\begin{rmk}
Condition \eqref{esigdiva} implies that 
$L_i(n)$ cannot have divisors in~$\Sigma$; in particular,
 ``$\Sigma$-squarefree'' 
from Theorem~\ref{thsieving} can be replaced by ``squarefree''.
\end{rmk}

We start from an individual prime. In the sequel, 
we write ${a=a_k}$, ${b=b_k}$ and 
${L(n)=L_k(n)=an+b}$. We also set ${M=\max\{|a|,|b|\}}$. 

\begin{proposition}
\label{poneprime}
Assume the hypothesis of Theorem~\ref{thsieving}. 
Further, assume that 
\begin{equation}
\label{epdandb}
\text{every prime from~$\Sigma$ divides~$a$}. 
\end{equation}
Then there exist real numbers
 ${C_3\ge 2}$ depending only on~$k$,  and ${z_3\ge 2}$ 
 depending on~$k$ and~$M$ such that the following holds. 
 Let~$p$ be a prime number, ${p\notin\Sigma}$. 
 Then for any~$x$ and~$z$ satisfying
${x\ge  z\ge z_3}$,
the set 
$\Omega_1(x,z)$ has at most 
${C_3 \cdot 2^{\#\Sigma} (x/p)(\log z)^{-k}}$ 
elements~$n$ such that $p\mid L(n)$. 
\end{proposition}

\begin{proof}
In this proof, every constant implied by $O(\cdot)$, $\ll$ etc. 
depends only on~$k$. 
We may assume that $L(n)$ is divisible by~$p$ 
for some ${n\in \Z}$ (otherwise there is nothing to prove).
 It follows that ${p\nmid a}$. (Indeed, if ${p\mid a}$ 
 then ${p\nmid b}$ because~$a$ and~$b$ are coprime, 
 and the congruence ${an\equiv- b \bmod p}$ is impossible.) 
Hence, there is a unique ${u\in \{0,1, \ldots, p-1\}}$ 
such that ${u\equiv -b/a\bmod p}$. 

For ${i=1,\ldots, k}$, set
$$
a_i'=
\begin{cases}
a_i, &p\mid L_i(u), \\
pa_i, &p\nmid L_i(u),
\end{cases}
\quad 
b_i'=
\begin{cases}
L_i(u)/p, &p\mid L_i(u), \\
L_i(u), &p\nmid L_i(u),
\end{cases}
$$
and write
$$
L_i'(n')= a_i'n'+b_i'.
$$
An immediate verification shows that \eqref{enondegeni},
 \eqref{enondegenij} and \eqref{esunit} remain true
 with~$a_i$,~ $b_i$ and~$\Sigma$ replaced by~$a_i'$,~$b_i'$ 
 and ${\Sigma'=\Sigma\cup\{p\}}$. Hence, defining 
 for ${x'\ge z'\ge 2}$,  the set 
\begin{equation*}
\Omega'(x',z')= \left\{0 \le n'\le x': \gcd\bigl(L_1'(n')\cdots L_k'(n'), 
~P_{\Sigma'}(z')\bigr)=1\right\},
\end{equation*}
we may apply Proposition~\ref{propbmho}: there exists~$z_0'$, 
depending 
only on~$k$ such that, when ${x'\ge (z')^{50k}}$ and 
${z'\ge z_0'}$, we have 
\begin{equation}
\label{emhoprime}
\#\Omega'(x',z')\ll2^{\#\Sigma'}\frac{x'}{(\log z')^k}
\ll2^{\#\Sigma}\frac{x'}{(\log z')^k}. 
\end{equation}

Every~$n$ with ${p\mid L(n)}$ can be written as ${u+n'p}$  
with ${n'\in \Z}$. If ${n\in \Omega(x,z)}$, then clearly
we have ${0\le n'\le x/p}$. Also 
$$
L_i(n)=
\begin{cases}
pL_i'(n'), &p\mid L_i(u), \\
L_i'(n'), &p\nmid L_i(u) 
\end{cases}
\qquad (i=1, \ldots, k). 
$$ 
It follows that the number of ${n\in \Omega(x,z)}$ 
such that ${p\mid L(n)}$ 
is bounded by ${\#\Omega'(x/p,z)}$.

Unfortunately, we cannot apply \eqref{emhoprime} with ${x'=x/p}$ 
and ${z'=z}$, because we do not have ${x'\ge (z')^{50k}}$. This is the 
main reason why we had to replace $\Omega(x,z)$ by $\Omega_1(x,z)$, 
because if ${n\in \Omega_1(x,z)}$ then we can bound $x/p$ from below. 

Indeed, let ${n\in \Omega_1(x,z)}$ be such that ${p\mid L(n)}$. 
By the definition of the set $\Omega_1(x,z)$, we know that 
$L(n)$ is composite and \eqref{epdandb} implies 
that $L(n)$ is not divisible by any 
primes from~$\Sigma$. Hence $L(n)/p$ must be divisible by 
some prime ${p'\ge z}$. In particular, ${|L(n)/p|\ge z}$, 
which implies that 
${x/p\ge z/M-1}$ 
(recall that ${M=\max\{|a|,|b|\}}$). Now setting ${x'=x/p}$ and 
${z'=(z/M-1)^{1/50k}}$, we obtain
\begin{align}
\#\{n\in \Omega_1(x,z):p\mid L(n)\}&\le \#\Omega'(x/p,z)\nonumber\\
&\le \#\Omega'(x',z')\nonumber\\
&\ll 2^{\#\Sigma}\frac{x'}{(\log z')^k}\nonumber\\
\label{eestim}
&\ll 2^{\#\Sigma}\frac{x/p}{(\log (z/M-1))^k}, 
\end{align}
provided 
\begin{equation}
\label{ezmmone}
(z/M-1)^{1/50k}\ge z_0'.
\end{equation} 
If we define ${z_3=\max \{M(z_0')^{50k}+M, 4M^2\}}$, 
then ${z\ge z_3}$ implies both \eqref{ezmmone} 
and ${z/M-1\ge z^{1/2}}$. Hence, 
the right-hand side of \eqref{eestim} is 
${O\bigl(2^{\#\Sigma}(x/p)(\log z)^{-k}\bigr)}$, as wanted.
\end{proof}

We will also need the following easy lemma. 

\begin{lemma}
\label{lloglog}
Let~$\cP$ be a set of prime numbers, 
${\eps\in (0,1/2]}$ and ${z_0\ge 2}$. 
Assume that for all ${t\ge z_0}$, we have 
${\pi_\cP(t)\le \eps\pi(t)}$. 
Then for ${x\ge z\ge z_0}$ we have 
$$
\sum_{\genfrac{}{}{0pt}{}{p\in \cP}{z\le p<x}}
\frac1p \ll \eps\log \left(\frac {\log x}{\log z}\right) +\eps,
$$
the implied constant being absolute. 
\end{lemma}

\begin{proof}
Using partial summation, we have
\begin{align*}
\sum_{\genfrac{}{}{0pt}{}{p\in \cP}{z\le p<x}}\frac1p&
~=\frac{\pi_\cP(x^-)}x-\frac{\pi_\cP(z^-)}z
+\int_z^x\frac{\pi_\cP(t)}{t^2}dt
~\ll \eps\log\left(\frac{\log x}{\log z}\right)+\eps,
\end{align*}
as wanted. 
\end{proof}

\begin{proof}[Proof of Theorem~\ref{thavoiding}]
Let~$\eta$,~$c_1$ and~$z_1$ be as in Theorem~\ref{thsieving}. 
Then for ${x\ge z_1^{1/\eta}}$, we have 
${\#\Omega_1\ge c_1x(\log x)^{-k}}$, 
where we denote ${\Omega_1=\Omega_1(x,x^\eta)}$. 

Now let~$\cP$ be a set of prime numbers, and let~$\Omega_2$ 
be the subset of~$\Omega_1$ consisting of ${n\in \Omega_1}$ such that 
some ${p\in \cP}$ divides ${L_1(n)\ldots L_k(n)}$. 
Also let~$z_3$ be as in Proposition~\ref{poneprime}.   
Define ${z_2\ge \max\{z_1,z_3\}}$ so large that for ${t\ge z_2}$, 
we have ${\pi_\cP(t)\le 2 \bar\delta(\cP) \pi(t)}$, and set 
${x_0=z_2^{1/\eta}}$.  Proposition~\ref{poneprime} and 
Lemma~\ref{lloglog} imply that for ${x\ge  x_0}$, we have 
\begin{align*}
\#\Omega_2&\ll \frac{x}{(\log x)^k}
\sum_{\genfrac{}{}{0pt}{}{p\in \cP}{x^\eta\le p<x}}\frac1p\\
&\ll  \bar\delta(\cP)\frac{x}{(\log x)^k} 
\left(\log \frac {\log x}{\log x^\eta}+1\right)\\
& \ll \bar\delta(\cP)\frac{x}{(\log x)^k}, 
\end{align*}
where the implicit constants depend on~$k$ and on~$\#\Sigma$. 

It follows that there exists ${\eps \in (0,1/2]}$, depending on~$k$ 
and on~$\#\Sigma$, such that, when 
${\bar\delta (\cP)\le \eps}$,  we have 
$$
\#\Omega_2 \le \frac12 c_1 \frac{x}{(\log x)^k} \le \frac12\#\Omega_1.
$$
This completes the proof of the theorem. 
\end{proof}

\section{Proof of Theorem~\ref{thnew} and Theorem~\ref{tcorold}}
\label{sproof}

Throughout the section, we assume that  
${f_1, \ldots, f_k}$ are primitive forms of square-free levels
(as defined in the beginning of Section~\ref{sintr}) of 
weights ${\kappa_1, \ldots, \kappa_k}$ 
respectively.  We also fix, once and for all, distinct positive 
integers ${\nu_1, \ldots, \nu_k}$ satisfying 
${a_{f_i}(\nu_i)\ne 0}$ for ${i=1, \ldots, k}$. 
We will assume that ${k \ge 2}$ as otherwise we
know that any non-zero primitive form
has infinitely many non-zero Fourier
coefficients (see Proposition \ref{lnup}).
Set ${K=\max\{\nu_1, \ldots, \nu_k\}}$.

\subsection{An application of the Chinese remainder theorem}

\begin{proposition}
\label{preplace}
Let $m \ge 1$ be such that
\begin{equation}
\label{ekdivn}
m \equiv 0\bmod{(2K)!}.
\end{equation}
There exists a positive real number $c_0$, depending 
on~$K$ and ${f_1, \ldots, f_k}$, such that for~$m$ 
satisfying \eqref{ekdivn} we have
\begin{equation}\label{nor}
c_0|\lambda_{f_i}(m_i)|
~\le~ |\lambda_{f_i}(m + \nu_i)|
~\le~  c_0^{-1}|\lambda_{f_i}(m_i)| \qquad (i=1,\ldots, k),
\end{equation}
and
\begin{equation}
\label{etaumi}
c_0 |\lambda_{f_i}(m_i)|
~\le~ \frac{|a_{f_i}(m + \nu_i)|}{m_{\vphantom{i}}^{(\ka_i -1)/2}}
~\le~  c_0^{-1}|\lambda_{f_i}(m_i)| \qquad (i=1,\ldots, k),
\end{equation}
where $m_i$ is defined by
\begin{equation}
\label{eni}
m + \nu_i= \nu_i m_i \qquad (i =1, \ldots, k).
\end{equation} 
\end{proposition}

\begin{proof}
It follows from \eqref{ekdivn} and the definition of $K$
that each $m_i$ is coprime to $(2K)!$. In particular,
\begin{equation}
\gcd(\nu_i,m_i)=1 \qquad (i=1,\ldots,k). 
\end{equation}
Since ${a_{f_i}(\nu_i)\ne 0}$ for ${i=1, \ldots, k}$, we may define 
$$
c_0 ~=~
\min_{1\le i\le k}\min\left\{ |\lambda_{f_i}(\nu_i)|,
~\frac{1}{2^{(\ka_i -1)/2}|\lambda_{f_i}(\nu_i)|}\right\}. 
$$
Hence, by multiplicativity, we have
$$
\frac{|a_{f_i}(m+\nu_i)|}{m^{(\ka_i -1)/2}}\ge
 |\lambda_{f_i}(m+\nu_i)| = 
|\lambda_{f_i}(\nu_i)|
|\lambda_{f_i}(m_i)| \ge 
c_0 |\lambda_{f_i}(m_i)| 
$$
and 
$$
 |\lambda_{f_i}(m+\nu_i)|
 =|\lambda_{f_i}(\nu_i)| |\lambda_{f_i}(m_i)| \le  
c_0^{-1} |\lambda_{f_i}(m_i)|.
$$
This completes the proof of \eqref{nor}.
Since ${m \ge 2K}$, we have 
$$
{(m + \nu_i)^{(\ka_i -1)/2}
\le 2^{(\ka_i-1)/2}m^{(\ka_i -1)/2}}
$$ 
and then again by multiplicativity, one has
\begin{align*}
\frac{|a_{f_i}(m+\nu_i)|}{m^{(\ka_i -1)/2}} 
&\le 2^{(\ka_i -1)/2} |\lambda_{f_i}(m+\nu_i)| \\
&= 2^{(\ka_i -1)/2}|\lambda_{f_i}(\nu_i)|
|\lambda_{f_i}(m_i)|\\
&\le c_0^{-1} |\lambda_{f_i}(m_i)|. 
\end{align*}
This completes the proof of \eqref{etaumi}.
\end{proof}

\subsection{Sieving and Sato-Tate}
 
Next we choose primes ${p_1<\cdots<p_k}$ with ${p_1> 2K}$ 
such that  
\begin{align}
\label{eallntwop}
\lambda_{f_i}(p_i)&\ne \pm2 && (i=1,\ldots,k),\\ 
\label{eallnze}
\lambda_{f_i}(p_i^\ell)&\ne 0 && (i=1,\ldots,k, \quad \ell=1,2,\ldots).
\end{align}
Existence of such primes is guaranteed by 
Propositions~\ref{prsatotate} and~\ref{prtaulnzer}. 

Let ${\ell_1, \ldots, \ell_k}$ be positive integers 
which will be specified later. We now impose on~$m$, 
besides \eqref{ekdivn}, the conditions 
\begin{equation}
\label{emmodplam}
m + \nu_i \equiv p_i^{\ell_i}\bmod p_i^{\ell_i+1} \qquad (i=1, \ldots, k). 
\end{equation}
Together with \eqref{ekdivn} this puts~$m$ into an arithmetic 
progression modulo~$A$, where 
\begin{equation}
\label{eq:7}
A=(2K)!\prod_{i=1}^k p_i^{\ell_i+1}. 
\end{equation}
Write ${m=An+B}$, where ${B<A}$ is the smallest 
positive integer in this progression. Here, ${n\ge 0}$ is 
some non-negative integer. Then 
${m+\nu_i=\nu_ip_i^{\ell_i} (a_in +b_i)}$, where 
\begin{equation}\label{new}
a_i=\frac A{\nu_ip_i^{\ell_i}}, \quad b_i
=
\frac{B+\nu_i}{\nu_ip_i^{\ell_i}} \qquad (i=1, \ldots, k)
\end{equation}
are positive integers\footnote{There is no risk of confusing the 
Hecke eigenvalues $a_{f_i}(n)$ and the integers~$a_i$.}. 
In particular, the numbers~$m_i$ defined in \eqref{eni}
are given by 
\begin{equation}
\label{emi}
m_i=p_i^{\ell_i} L_i(n), \qquad (i=1, \ldots, k), 
\end{equation}
where ${L_i(n)=a_in+b_i}$.

Note that
\begin{align*}
&\gcd(A,B+\nu_i)=\nu_ip_i^{\ell_i}
&& (i=1, \ldots, k),\\
&
a_ib_j-a_jb_i
=
\frac A{\nu_i\nu_jp_i^{\ell_i}p_j^{\ell_j}}(\nu_j-\nu_i)
&&(1\le i,j\le k). 
\end{align*}
In particular, it follows that the integers
$a_1, \ldots, a_k, b_1, \ldots, b_k$ defined in \eqref{new} 
satisfy  \eqref{enondegeni} and \eqref{enondegenij}. Moreover, setting 
$$
\Sigma=\{p\le 2K\}\cup\{p_1, \ldots, p_k\},
$$
conditions \eqref{esunit} and \eqref{esigdiva} hold true as well, 
which allows us to apply our sieving Theorems~\ref{thsieving} 
and~\ref{thavoiding}. Using them  and the Sato-Tate conjecture 
(as stated in Proposition~\ref{prsatotate}), we obtain the following.

\begin{proposition}
\label{plis}
There exists a positive number~$c_1$, depending on~$K$
and on the forms~$f_i$ such that there exist infinitely many positive 
integers~$n$ with the following property
\begin{equation}
\label{etauli}
c_1~\le~ |\lambda_{f_i}(L_i(n))|
~\le~ c_1^{-1} \qquad (i=1, \ldots, k). 
\end{equation}
\end{proposition}

\begin{proof}
Let~$\eta$ and $\eps$ be as in Theorem~\ref{thsieving} 
and Theorem~\ref{thavoiding}
respectively. Both depend on~$k$ and~$\#\Sigma$, but since 
${\#\Sigma=\pi(2K)+k}$, this translates into dependence on~$K$.

Now let~$\cP_\eps$ be the set of prime numbers~$p$ 
such that for some ${i\in \{1, \ldots, k\}}$ we have 
${|\lambda_{f_i}(p)| \le \eps/k}$. 
Proposition~\ref{prsatotate} implies that its relative density 
is at most~$\eps$. Now Theorem~\ref{thsieving} and 
Theorem~\ref{thavoiding} 
together imply that there exist infinitely many  positive integers~$n$ 
with the following properties:

\begin{enumerate}

\item
each $L_i(n)$ is a square-free positive 
integer;

\item
\label{ilower}
for ${i=1, \ldots, k}$, every prime ${p\mid L_i(n)}$ 
satisfies ${p\ge n^{\eta}}$; 

\item
for ${i=1, \ldots, k}$,  every prime ${p\mid L_i(n)}$ satisfies 
$|\lambda_{f_i}(p)| > \eps/k$.

\end{enumerate}

After discarding finitely many numbers~$n$, 
item~\ref{ilower} implies that 
\begin{enumerate}
\item[\ref{ilower}${}'$.]
for ${i=1, \ldots, k}$, every prime ${p\mid L_i(n)}$ 
satisfies ${p\ge L_i(n)^{\eta/2}}$. 
\end{enumerate}
Hence, each $L_i(n)$ has at most ${2/\eta}$ prime divisors. 
Write ${L_i(n)=q_1\cdots q_s}$, where ${s\le 2/\eta}$ and 
${q_1, \ldots, q_s}$ are distinct prime numbers satisfying 
$$
\frac\eps k 
~<~  |\lambda_{f_i}(q_j)|
~\le~ 2 \qquad (j=1,\ldots, s). 
$$
The inequality on the right is by Deligne's bound \eqref{edel}. 
By multiplicativity, we now obtain 
$$
\left(\frac\eps k\right)^{2/\eta}
~\le~ 
|\lambda_{f_i}(L_i(n))|
~\le~ 2^{2/\eta}.
$$
This completes the proof.
\end{proof}

\begin{remark}
\label{resatotate}
Slightly modifying the above argument, one proves the following 
quantitative result: there exist ${c_2>0}$ (depending on~$K$) and  
 ${x_0\ge2}$ (depending on~$K$, on the forms~$f_i$ and on our choice of the 
 primes~$p_i$ and the exponents~$\ell_i$) such that for ${x\ge x_0}$ the number of 
${n\le x}$ with the property \eqref{etauli} is at least $c_2 x(\log x)^{-k}$. 
The constant~$c_2$  is effective, but~$x_0$ is not,  
because it depends on a ``quantitative'' form of the Sato-Tate conjecture, 
which is not known to be effective (to the best of our knowledge). 
\end{remark}

\subsection{The Exponents~$\ell_i$}
We now fix a small parameter~$\delta >0$ (to be specified later) 
and define, in terms of this~$\delta$,  our ${\ell_1, \ldots, \ell_k}$. 

\begin{proposition}
\label{plambdas}
Let~$\delta$ be a positive real number. Then there exist positive 
integers ${\ell_1, \ldots, \ell_k}$ such that 
\begin{equation}
\label{elambdas}
\left|\lambda_{f_1}\bigl(p_1^{\ell_1}\bigr)\right| 
< \delta \left|\lambda_{f_2}\bigl(p_2^{\ell_2}\bigr)\right|
<\ldots 
<\delta^{k-1}\left|\lambda_{f_k}\bigl(p_k^{\ell_k}\bigr)\right|.
\end{equation}
\end{proposition}

We start with an easy lemma.

\begin{lemma}\label{nle}
Let~$f$ be a primitive form of weight~$\kappa$, let~$p$ be 
a prime number such that ${\lambda_f(p)\ne \pm2 }$, 
and let~$\eps$ a positive real number. Then there 
exists a positive integer~$\ell$ such that 
${|\lambda_f(p^\ell)|<\eps}$.
\end{lemma}

\begin{proof}
We may assume ${\theta_p/\pi\notin\Q}$ as otherwise there 
is nothing to prove. 
Using \eqref{esines}, we know that
$$
|\lambda_f(p^\ell)|
= 
\frac{|\sin((\ell+1)\theta_p)|}{|\sin\theta_p|}. 
$$
Since ${\theta_p/\pi\notin\Q}$,  selecting~$\ell$ suitably, 
we can make ${|\sin((\ell+1)\theta_p)|}$ as small as we please. 
\end{proof}

\begin{corollary}
\label{colmulam}
Let $f,g$ be  primitive forms of weights $\kappa,\rho$, 
respectively, and let $p, q$  be prime numbers.
Also let~$\ell'$ be a positive integer and~$\delta$ be a 
positive real number. Assume that ${\lambda_f(p)\ne \pm2}$ 
and ${a_g(q^{\ell'})\ne 0}$.  Then there 
exists a positive integer~$\ell$ such that 
$$
|\lambda_f(p^\ell)|
~<~ 
\delta |\lambda_g(q^{\ell'})|.
$$
\end{corollary}

\begin{proof}
Apply
Lemma~\ref{nle} with 
${\eps =\delta|\lambda_g(q^{\ell'})|}$. 
\end{proof}

\begin{proof}[Proof of Proposition~\ref{plambdas}]
Set ${\ell_k=1}$ and afterwards define 
${\ell_{k-1}, \ldots, \ell_1}$ 
iteratively by applying Corollary~\ref{colmulam} $(k-1)$-times. The hypothesis of 
Corollary~\ref{colmulam} is assured because of \eqref{eallntwop} and \eqref{eallnze}. 
\end{proof}

\begin{remark}
Using Baker's theory of logarithmic forms, it is possible to 
prove that one can find suitable ${\ell_1, \ldots, \ell_k}$ 
effectively bounded in terms of ${f_1, \ldots, f_k}$ and~$\delta$. We do not go into 
details since we do not need this. 
\end{remark}

\subsection{Conclusion}

Now we are ready to prove Theorem~\ref{thnew}
and Theorem~\ref{tcorold}. Let~$c_0$ 
and~$c_1$ be as in Proposition~\ref{preplace} and 
Proposition~\ref{plis} respectively. 
Set ${\delta=(c_0c_1)^2/ 2}$ and 
define the exponents ${\ell_1, \ldots, \ell_k}$ as in 
Proposition~\ref{plambdas}. (It is crucial here that 
~$c_0$ and~$c_1$ depend only on~$K$
but not on the exponents $\ell_i$.)  
Now if~$n$ is one of the infinitely many positive integers 
satisfying property \eqref{etauli}, then in the set-up of Theorem~\ref{thnew} the corresponding 
${m=An+B}$ satisfies 
\begin{equation*}
|\lambda_{f_1}(m+\nu_1)|\le \frac12|\lambda_{f_2}(m+\nu_2)|\le\cdots 
\le 
\frac1{2^{k-1}}|\lambda_{f_k}(m+\nu_k)|
\end{equation*}
as follows from \eqref{nor}, \eqref{emi}, 
\eqref{etauli} and \eqref{elambdas}. 
In the set-up of Theorem~\ref{tcorold} it satisfies
\begin{equation*}
|a_{f}(m+\nu_1)|\le \frac12|a_{f}(m+\nu_2)|\le\cdots 
\le 
\frac1{2^{k-1}}|a_{f}(m+\nu_k)|, 
\end{equation*}
as follows from \eqref{etaumi} (with ${f_1=\ldots=f_k=f}$), \eqref{emi}, \eqref{etauli} 
and \eqref{elambdas}. This completes the proof
of Theorem~\ref{thnew} and Theorem~\ref{tcorold}.

\begin{remark}
\label{requant}
As Remark~\ref{resatotate} suggests, we actually obtain the following 
quantitative results: for sufficiently large~$x$, there is at least 
${cx(\log x)^{-k}}$ positive integers ${m\le x}$ with the 
property \eqref{etausnu} and \eqref{etausnu1}
; here ${c=c(K, f_1, \ldots, f_k)>0}$ is effective 
and ``sufficiently large'' is not effective. 
\end{remark}

\section{Proof of Theorem~\ref{thdes}}
\label{sdes}

In this section~$k$ is a positive integer, and~$f$ is a 
primitive form of square-free level, 
as defined in the beginning of Section~\ref{sintr}.   
We want to show that the three conditions 
\ref{inezhalf}, \ref{inezk} and \ref{iperm} are equivalent. 
We will assume that ${k \ge 2}$ as otherwise we
know that any non-zero primitive form
has infinitely many non-zero Fourier
coefficients (see Proposition \ref{lnup}).
Condition~\ref{iperm} trivially implies~\ref{inezk}, and~\ref{inezk} 
implies~\ref{iperm} by putting 
$$
\nu_i=\nu+\sigma(i)\qquad (1\le i\le k)
$$
in Theorem~\ref{tcorold}. 

The implication \ref{inezk}$\Rightarrow$\ref{inezhalf} is easy. One 
readily sees that \eqref{etnezhalf} is equivalent to the following:
\begin{equation}\label{eCNSbis}
\text{$a_f\left(p^\ell\right) \neq 0$  for every prime $p$ and 
positive $\ell$
with  $p^\ell \le k/2$.}
\end{equation}
We will check \eqref{eCNSbis}; let $p$ and $\ell$ be such that
$p^\ell \le k/2$. Since ${k \ge 2p^\ell}$,
the set $\{\nu+1, \ldots, \nu+k\}$ contains at least two consecutive multiples 
of~$p^\ell$ and so one of them, say ${\nu+h}$, is divisible by $p^\ell$ but not
by $p^{\ell+1}$. Since $a_f$ is multiplicative and ${a_f(\nu+h) \neq 0}$, 
we have ${a_f(p^\ell)\neq 0}$.

We are left with the implication \ref{inezhalf}$\Rightarrow$\ref{inezk}. 
We deduce it from Theorems~\ref{thsieving} and~\ref{thavoiding} 
with the help of the following lemma.

\begin{lemma}
\label{lnup}
Let~$f$ be a primitive form of square-free level
$N$. For every prime number~$p$ there exist infinitely 
many integers $\ell$ such that
\begin{equation*}
a_f\left(p^\ell\right) \neq 0.
\end{equation*}
\end{lemma}

\begin{proof}
If $p| N$, then we know from the Atkin-Lehner theory
that 
\begin{equation}\label{AL}
a_f\left(p^\ell\right) = a_f(p)^{\ell} \neq 0
\end{equation}
as $N$ is square-free
(see page 29 of \cite{KO}). We shall now only consider 
primes $p$ with $(p,N)=1$.
We shall indeed prove that among two consecutive
non-negative integers
${(\ell, \ell+1)}$, at least one, say~${\ell'}$, satisfies
${a_f\left(p^{{\ell'}}\right) \neq 0}$.

Our claim is true for $\ell=0$ since ${a_f(1)=1}$. Let 
us assume (induction hypothesis) that it is true 
for a pair $(\ell, \ell+1)$.

If ${a_f\left(p^{\ell+1}\right)\neq 0}$, then our claim is true for the pair 
${(\ell + 1, \ell+2)}$.
On the other hand, if 
${a_f\left(p^{\ell+1}\right)= 0}$, then  ${a_f\left(p^{\ell}\right)\neq 0}$
by our induction hypothesis,
and~\eqref{erecu} implies that
$$
a_f\left(p^{\ell+2}\right) = a_f(p) 
a_f\left(p^{\ell+1}\right) - p^{\kappa-1} 
a_f\left(p^{\ell}\right) = - p^{\kappa-1} a_f\left(p^{\ell}\right) \neq 0.
$$
Hence, our claim is again true for the pair $(\ell + 1, \ell+2)$. 
This proves the lemma.
\end{proof}

Alternatively, it is possible to deduce the lemma from  
equations~\eqref{esines},~\eqref{ezeropi} and~\eqref{AL}; 
we leave the details to the reader. 

\paragraph{Proof of the implication \ref{inezhalf}$\Rightarrow$\ref{inezk}.}

We assume that \eqref{eCNSbis} holds and want to find a positive 
integer~$\nu$ such that \eqref{etnezk} holds. 

Since~(\ref{eCNSbis}) is the same when ${k=2 h}$ and ${k=2 h +1}$, namely 
${a_f\left(p^{\ell}\right) \neq 0}$ for  ${p^{\ell} \le h}$, it is sufficient 
to consider the case when~$k$ is odd, say ${k=2 h +1}$.

We define~$\Sigma$ as the set of all primes ${p\le 2k}$
and those finitely many primes $p$ for which $a_f(p) \ne 0$
but $a_f(p^{\ell})=0$ for some $\ell >1$. 
By Lemma~\ref{lnup}, to each ${p\in\Sigma}$ we may associate 
an integer~$\ell_p$ such that
\begin{align}
\label{epnup}
a_f\left(p^{\ell_p}\right)&\neq 0, \\
\label{egrek}
p^{\ell_p} &> k.
\end{align}
By the Chinese remainder theorem, one can find a positive integer $r$ such that
\begin{equation}
\label{er}
r \equiv p^{\ell_p} \bmod p^{\ell_p+1} \qquad (p\in \Sigma). 
\end{equation}
We will show that there exist infinitely many positive integers~$m$ such that
\begin{equation}
\label{enonzero}
a_f(Dm+r+j)\neq 0 \qquad (-h \le j \le h), 
\end{equation}
where 
$$
D = \prod_{p \in \Sigma} p^{\ell_p +1}.
$$
If~$m$ is any such integer, then, setting ${\nu=r+Dm-h-1}$, we clearly obtain \eqref{etnezk}.

For ${-h \le j \le h}$, we introduce the linear forms ${L_j(m)=a_jm+b_j}$ by
\begin{equation*}
Dm+r+j = \gcd(D, r+j)L_j(m) = \gcd(D, r+j) (a_jm+b_j).
\end{equation*}
(There is no risk of confusing the Hecke eigenvalues $a_{f}(n)$ and the integers~$a_j$.) Let us first check that the~$k$ linear forms $L_j$ satisfy the conditions 
of Theorem~\ref{thsieving} and Theorem~\ref{thavoiding}.

\begin{itemize}
\item
By construction, for every~$j$, we have ${a_j \neq 0}$ and ${\gcd(a_j, b_j) =1}$.

\item
For ${i\neq j}$, we have ${D(r+j)-D(r+i) = D(j-i) \neq 0}$. Since ${a_i b_j - a_j b_i}$ 
is a divisor of ${D|j-i|}$, it is not~$0$. 

\item
By construction,~$a_i$ is a divisor of~$D$ which has only prime divisors 
from~$\Sigma$.

\item
Similarly, ${a_i b_j - a_j b_i}$ is a divisor of ${D|j-i|}$, where~$D$ and ${|j-i| \le k}$ 
have only prime divisors from~$\Sigma$.

\item
We finally have to verify that every $a_j$ is divisible by every prime in the set~$\Sigma$. Since ${r \equiv p^{\ell_p} \bmod p^{\ell_p+1}}$
and ${p^{\ell_p} > k > h}$,  we have ${\ord_p(r+j) \le \ell_p}$ 
(where $\ord_p$ denotes the $p$-adic valuation). Now since 
$$
\ord_p(a_j) = \ord_p(D) - \ord_p(r+j)
$$
and  ${p^{\ell_p+1} \mid D}$, we have ${\ord_p(a_j) \ge 1}$.

\end{itemize}

We can now apply Theorems \ref{thsieving} and \ref{thavoiding}, 
taking for the 
unwanted set of primes those which are not in $\Sigma$ and for which 
$a_f(p)=0$. Thus, there exist infinitely many positive integers~$m$ 
such that each of the~$k$ numbers
${L_{-h}(m),\ldots, L_{h}(m)}$ is square-free, not divisible by any prime from~$\Sigma$ 
nor by any prime~$p$ for which ${a_f(p)=0}$. It follows that for such~$m$, we have
$$
a_f(L_j(m))\ne 0 \qquad(-h\le j\le h). 
$$
In order to prove that for these~$m$
we have \eqref{enonzero}, it is enough to prove that 
\begin{equation}
\label{enzergcd}
a_f(\gcd(D, r+j))\neq 0 \qquad(-h \le j\le h).
\end{equation}
When ${j =0}$, for any~$p$ in~$\Sigma$ we have ${p^{\ell_p} \,\|\, r}$
so that ${p^{\ell_p} \,\|\, \gcd(D, r)}$. Since 
${a_f(p^{\ell_p})\ne 0}$ by \eqref{epnup}, we obtain ${a_f(\gcd(D, r)) \neq 0}$
 by multiplicativity. 

If ${j \neq 0}$, then,  for ${p \in \Sigma}$ we have ${\ord_p(j)<\ell_p}$ 
because ${p^{\ell_p}>k}$ by \eqref{egrek}. 
Hence
${p^{\mu} \,\|\, \gcd(D, r+j)}$ implies that  ${p^{\mu} \,\|\,j}$.  It follows that
${p^{\mu} \le h \le k/2}$, and our assumption \eqref{etnezhalf} 
implies that ${a_f(p^{\mu})\neq 0}$. By multiplicativity, this 
proves \eqref{enzergcd} for ${j\ne 0}$ as well. 

The proof of the implication \ref{inezhalf}$\Rightarrow$\ref{inezk} 
is now complete, and so is the proof of Theorem~\ref{thdes}. 
\qed

\paragraph{Acknowledgements.} The authors would like to thank
the referee for her/his extremely careful reading and relevant suggestions
which improved the exposition of the paper.

F.~L. worked on this paper during a 
visit to the Institute of Mathematics of Bordeaux as an ALGANT 
scholar in July 2011. He thanks ALGANT for support 
and the French Ministry of Defence for allowing him, after some time, to 
enter the IMB building. In addition, this author was also partially 
supported by grant CPRR160325161141 and an A-rated scientist award 
both from the NRF of 
South Africa, by grant no. 17-02804S of the Czech Granting Agency and by IRN ``GANDA'' (CNRS). 

Yu.~B. was partially supported by the ALGANT, by the  Indo-European 
Action Marie Curie (IRSES moduli) and by IRN ``GANDA'' (CNRS).

J.-M.~D. was partially supported by the CEFIPRA project 5401-A, by the  
Indo-European Action Marie Curie (IRSES moduli) 
and by IRN ``GANDA'' (CNRS).

S.~G.  was visiting the  Institute of Mathematics of 
Bordeaux as an ALGANT scholar in 2014 when 
she started working on this project. She acknowledges 
support by the ALGANT, a SERB grant and the DAE number
 theory plan project.

During the final stage of preparation of this paper Yu.~B. 
and  J.-M.~D. enjoyed hospitality of the Institute of 
Mathematical Sciences at Chennai.

The authors thank Satadal Ganguly, Florent Jouve and 
Gérald Tenenbaum for  helpful advice.

{\footnotesize

\bigskip

\noindent
\textsc{Yuri ~F.~Bilu}\newline
{Université de Bordeaux and CNRS}\newline
{Institut de Mathématiques de Bordeaux UMR 5251}\newline
{33405, Talence, France}; \newline
\textsf{yuri@math.u-bordeaux.fr} 

\bigskip

\noindent
\textsc{Jean-Marc~Deshouillers}\newline
Univ. Bordeaux, CNRS and Bordeaux INP\newline
Institut de Mathématiques de Bordeaux UMR 5251\newline
F-33405, Talence, France;\newline
\textsf{jean-marc.deshouillers@math.u-bordeaux.fr} 

\bigskip

\noindent
\textsc{Sanoli ~Gun}\newline
{Institute of Mathematical Sciences, HBNI}\newline
{C.I.T Campus, Taramani}\newline
{Chennai  600 113, India};\newline
\textsf{sanoli@imsc.res.in}

\bigskip

\noindent
\textsc{Florian~Luca}\newline
{School of Mathematics}\newline
{University of the Witwatersrand}\newline
{Private Bag X3, Wits 2050, South Africa};

\smallskip\noindent
{Department of Mathematics}\newline
{Faculty of Sciences}\newline
{University of Ostrava}\newline
{30. dubna 22, 701 03 Ostrava 1, Czech Republic};\newline
\textsf{florian.luca@wits.ac.za}

}

\end{document}